\documentclass[11pt,A4paper]{amsart}
\usepackage{amsfonts, amsmath, amssymb, amsthm, color}

\usepackage{hyperref}
\hypersetup{
  pdfborder={0 0 0},
  colorlinks   = true, 
  urlcolor     = blue, 
  linkcolor    = blue, 
  citecolor   = red 
}

\newtheorem{thm}{Theorem}[section]
\newtheorem{prop}[thm]{Proposition}
\newtheorem{lemma}[thm]{Lemma}
\newtheorem{cor}[thm]{Corollary}

\theoremstyle{definition}

\newtheorem{question}{Question}

\theoremstyle{remark}

\newtheorem{rem}[thm]{Remark}

\newcommand{\ga}{\Gamma}
\newcommand{\pc}{{\rm Pc}}
\newcommand{\C}{{\rm C}}
\newcommand{\No}{{\rm N}}

\DeclareMathOperator{\link}{link}
\DeclareMathOperator{\supp}{supp}
\DeclareMathOperator{\cd}{cd}

\newcommand{\cA}{\mathcal{A}}

\newcommand{\bF}{\overline{F}}
\newcommand{\bG}{\overline{G}}

\newcommand{\N }{\mathbb N}
\newcommand{\Z }{\mathbb Z}
\newcommand{\F}{\mathbb{F}}

\def\coloneq{\mathrel{\mathop\mathchar"303A}\mkern-1.2mu=}

\begin{document}

\title{On subgroups of right angled Artin groups with few generators}

\author{Ashot Minasyan}
\address{Mathematical Sciences,
University of Southampton, Highfield, Southampton, SO17 1BJ, United
Kingdom.}
\email{aminasyan@gmail.com}
%

\begin{abstract} For each $d \in \N$ we construct a $3$-generated group $H_d$, which is a subdirect product of free groups, such that the cohomological dimension of $H_d$ is $d$.
Given a group $F$ and a normal subgroup $N \lhd F$ we prove that any right angled Artin group containing the special HNN-extension of $F$ with respect to $N$ must also contain $F/N$.
We apply this to construct, for every $d \in \N$, a $4$-generated group $G_d$, embeddable into a right angled Artin group, such that the cohomological dimension of $G_d$ is $2$ but
the cohomological dimension of any right angled Artin group, containing $G_d$, is at least $d$. These examples are used to show the non-existence of certain ``universal'' right angled Artin groups.

We also investigate finitely presented subgroups of direct products of limit groups. In particular we show that for every $n\in \N$ there exists $\delta(n) \in \N$ such that any $n$-generated finitely presented subgroup of a direct
product of finitely many free groups embeds into the $\delta(n)$-th direct power of the free group of rank $2$. As another corollary we derive that any $n$-generated finitely presented residually free group embeds into
the direct product of at most $\delta(n)$ limit groups.
\end{abstract}


\maketitle

\section{Introduction}
A \emph{right angled Artin group}, also called a \emph{graph group} or a \emph{partially commutative group} in the literature, is a group which has a finite presentation, where the only permitted defining
relators are commutators of the generators. To get such a group, one normally starts with any finite simplicial graph $\ga$, with vertex set $V\ga$ and edge set $E\ga$, and defines the associated right angled
Artin group $A=A(\ga)$ by the presentation $$A=\langle V\ga \,\|\, [u,v]=1, \, \forall \{u,v\}\in E\ga \rangle.$$

The structure of these groups and their subgroups has been a subject of intensive study in the recent years. The results of Haglund, Wise and Agol \cite{H-W,Wise-QH,Agol} show that many
previously studied groups (e.g., one-relator groups with torsion, limit groups, fundamental groups of closed hyperbolic $3$-manifolds) virtually embed into right angled Artin groups. Thus the class of subgroups of
right angled Artin groups is rather rich. However, the first theorem about right angled Artin groups, proved by Baudisch in \cite{Baud}, asserts that a $2$-generated\footnote{We say that a group $G$ is \emph{$n$-generated},
for some $n \in \N$, if $G$ has a generating set of cardinality at most $n$.}  subgroup of a right angled Artin group is either free or free abelian.
In particular, there exists a single right angled Artin group (e.g., $\Z*\Z^2$) which contains all two-generated subgroups of right angled Artin groups. In \cite{C-K} Casals-Ruiz and Kazachkov ask whether a similar fact is true for
$n$-generated subgroups of right angled Artin groups:

\begin{question}[{\cite[Question 1]{C-K}}]\label{q:0} Given $n \in \N$, does there exist a ``universal'' right angled Artin group which contains all $n$-generated subgroups of arbitrary right angled Artin groups?
\end{question}

Let $\cA$ denote the class of all groups which are isomorphic to subgroups of (finitely generated) right angled Artin groups.
Motivated by Question \ref{q:0}, in this note we show that there are no ``universal'' right angled Artin groups for various subclasses of the class $\cA$.
We start with the following statement, where $\mathbb{F}_2$ denotes the free group of rank $2$, $\mathbb{F}_2^d$ denotes its $d$-th direct power and $\cd(H)$ denotes the cohomological dimension of a group $H$
over the integers.

\begin{prop}\label{prop:H_d} For every $d \in \N$ there exists a $3$-generated subgroup $H_d \leqslant \mathbb{F}_2^d$ such that $H_d$ contains the free abelian group of rank $d$ as a subgroup and $\cd(H_d)=d$.
\end{prop}

Since $\mathbb{F}_2^d$ is a right angled Artin group, Proposition \ref{prop:H_d} shows that the answer to  Question~\ref{q:0} is negative already for $n=3$.
Indeed any group containing all $H_d$, $d \in \N$, must have infinite cohomological dimension, but it is well-known
that the cohomological dimension of a right angled Artin group $A$, associated to a finite graph $\ga$,
is finite and  equals to the clique number of $\ga$\footnote{This can be derived from the fact that $A$ is isomorphic to the fundamental group of a non-positively curved Salvetti cube complex whose dimension is precisely
the clique number of $\ga$ -- see \cite[Sec. 3.6]{Charney}.} (in particular, $\cd(A) \le |V|$).
In other words, $$\cd(A)=\max\{|U| \mid U \subseteq V\ga \mbox{ such that } \{u,v\} \in E\ga \mbox{ for all distinct } u,v \in U\}.$$

It is now natural to wonder whether imposing a bound on the cohomological dimension, in addition to the bound on the number of generators, would help. Namely, one can ask whether for every $n \in \N$  there is a single right
angled Artin group which contains all $n$-generated groups $G \in \cA$ with $\cd(G) \le n$.  To show that this is not the case, we employ special HNN-extensions.

Let $F$ be a group and let $N \leqslant F$ be any subgroup. The \emph{special HNN-extension of $F$ with respect to $N$} is the group $G$ defined by the (relative) presentation
\begin{equation}\label{eq:spec-HNN}
G=\langle F,t \,\|\, tht^{-1}=h,~\forall\,h \in N \rangle.
\end{equation}

If $G$ and $A$ are groups then we will write $G \hookrightarrow A$ to say that $G$ can be
isomorphically embedded into $A$.

\begin{thm} \label{thm:rope} Suppose that $F \in \cA$ and $N \lhd F$ is a normal subgroup such that $F/N \in \cA$. Let $G$ be the special HNN-extension of $F$ with respect to $N$, defined by presentation \eqref{eq:spec-HNN}.
Then
\begin{itemize}
  \item[(i)] $G \in \cA$;
  \item[(ii)] if $A$ is any right angled Artin group such that $G \hookrightarrow A$ then $F/N \hookrightarrow A$.
\end{itemize}
\end{thm}

The main technical result, used in the proof of Theorem \ref{thm:rope}, is Proposition \ref{prop:emb_quot} below, which implies that any right angled Artin group containing the special
HNN-extension \eqref{eq:spec-HNN}, with $N \lhd F$, must also contain the quotient $F/N$.

Given any $d \in \N$, let us apply Theorem \ref{thm:rope} to the case when $F$ is the free group of rank $3$ and $N \lhd F$  is the normal subgroup such that $F/N \cong H_d$, where $H_d$ is the group from Proposition~\ref{prop:H_d}.
Then the special HNN-extension $G_d$, of $F$ with respect to $N$, will be generated by $4$ elements and will belong to the class $\cA$ by Theorem \ref{thm:rope}.
And for any right angled Artin group $A$ with $G_d \hookrightarrow A$, we will have
$H_d \hookrightarrow A$. In particular, $\cd(A) \ge \cd(H_d)=d$. On the other hand, we see that $\cd(G_d)=2$ because $G_d$ is not free and acts on a tree with free vertex stabilizers (cf. \cite[Ex. 4 in Ch. VIII.2]{Brown}).
Thus we obtain the following corollary:

\begin{cor}\label{cor:G_d} For every $d \in \N$ there exists a $4$-generated group $G_d \in \cA$ such that $\cd(G_d)=2$ but for any right angled Artin group $A$, containing a copy of $G_d$, one has $\cd(A) \ge d$.
\end{cor}

The above corollary shows there there exists no ``universal'' right angled Artin group containing all  $4$-generated groups from $\cA$ of cohomological dimension $2$.

The next naturally arising question could ask whether for a given $n \in \N$ there are right angled Artin groups containing copies of all finitely presented $n$-generated groups $G \in \mathcal A$.
This question seems to be more subtle. The groups $H_d$ from Proposition \ref{prop:H_d} are not finitely presented (see Remark \ref{rem:not_fp} below), and a result of Bridson and Miller \cite{B-M} implies that
finitely presented subdirect products of free groups, intersecting each of the factors non-trivially, are virtually surjective on pairs, which shows that they are quite scarce.
In \cite{B-H-M-S} Bridson, Howie, Miller and Short extended this result to finitely presented subgroups in direct products of finitely generated fully residually free groups (a.k.a. limit groups).
We use this generalization to prove the following theorem.

\begin{thm} \label{thm:adeq_free} There exists a computable function $\delta: \N \to \N$ such that the following holds.
Suppose that $F_1,\dots,F_d$ are limit groups and $H \leqslant F_1 \times \dots \times F_d$ is a finitely presented subgroup such that $H \cap F_i \neq \{1\}$ for each $i=1,\dots,d$.
If $H$ is $n$-generated then $d \le \delta(n)$.
\end{thm}

In Remark \ref{rem:delta_estimate} we give an estimate $\delta(n) < 2^{18n^3}$ for all $n \in \N$. This estimate is not sharp, and it would be interesting to see whether there exists a polynomial (or even a linear)
upper bound for $\delta(n)$ in terms of $n$.

Theorem \ref{thm:adeq_free} implies that no finitely presented analogues of the groups $H_d$ from Proposition~\ref{prop:H_d} can be found among subgroups of direct product of free groups:

\begin{cor} \label{cor:univ_for_free} Given any $n \in \N$, every $n$-generated finitely presented subgroup $H$ of a direct product of finitely many free groups can be embedded into $\F_2^{\delta(n)}$.
In particular, $\cd(H) \le \delta(n)$.
\end{cor}

We would also like to mention another corollary of Theorem \ref{thm:adeq_free} concerning residually free groups, which may be of independent interest.
It is known that every finitely generated residually free group $G$ can be embedded in the direct product of finitely  many limit groups (see Baumslag, Myasnikov and Remeslennikov \cite{B-M-R},
Kharlampovich and Myasnikov \cite{Kh-M}, or Sela \cite{Sela}).
In the case when $G$ is finitely presented Theorem~\ref{thm:adeq_free} can be used to give a bound on the number of factors:

\begin{cor}\label{cor:res_free} Let $G$ be an $n$-generated  finitely presented residually free group, for some $n \in \N$. Then $G$ can be embedded into the direct product of at most $\delta(n)$ limit groups.
\end{cor}

Using the commutative transitivity of limit groups it is easy to show that for any $d \in \N$  the residually free group $H_d$ from Proposition \ref{prop:H_d} cannot be embedded into the direct product of less than $d$ limit groups.
Thus it is not possible to remove the assumption that $G$ is finitely presented in Corollary \ref{cor:res_free}.

Unfortunately the characterization of finitely presented subdirect product of free groups given by Bridson and Miller in \cite{B-M} does not extend to finitely presented subgroups of more general right angled Artin groups.
And while Theorem \ref{thm:rope} can be used to construct some finitely presented examples (see Corollary \ref{cor:fp_ex} in Section \ref{sec:5}), the following question remains open.

\begin{question}\label{q:1} Do there exist finitely presented analogues of groups $H_d$ from Proposition \ref{prop:H_d}? More precisely, does there exist $n \in \N$ such that for any $d \in \N$
there is an $n$-generated finitely presented group $Q_d \in \cA$ which contains the free abelian subgroup of rank $d$ (and/or which has $\cd(Q_d)=d$)?
\end{question}

A celebrated theorem of Higman, Neumann and Neumann \cite{H-N-N} states that any countable group $H$ can be embedded into a $2$-generated group $G$; moreover, one can take $G$
to have the same number of defining relations as $H$. It is natural to ask if there is an analogue of this statement within the class $\cA$:

\begin{question}\label{q:2} Does there exist $n \in \N$ such that every group $H \in \cA$ can be embedded into some $n$-generated group $G \in \cA$? Can one choose such $G$ to be finitely presented,
in addition to being generated by $n$ elements?
\end{question}

Note that by the theorem of Baudisch \cite{Baud}, mentioned above, the number $n$ in Question \ref{q:2} must be at least $3$.

\section{Proof of Proposition \ref{prop:H_d}}
\begin{proof}[Proof of Proposition \ref{prop:H_d}] Let $F=F(x,y,z)$ be the free group of rank $3$ with the free generating set $\{x,y,z\}$. Given any $d \in \N$, choose any elements $w_1,\dots,w_d \in \langle x,y \rangle\leqslant F$
which freely generate a free subgroup of rank $d$ in $\langle x,y \rangle \cong\mathbb{F}_2$. For each $i=1,\dots,d$, define the normal subgroup $L_i \lhd F$ as the normal closure of the element $z^{-1}w_i$, and let
$$K_i\coloneq \bigcap_{j=1, j \neq i}^d L_j \lhd F. $$
By construction, for every $i$ the quotient $F/L_i$ is canonically isomorphic to the rank $2$ free subgroup $\langle x,y \rangle \leqslant F$.
Observe that $K_d \setminus L_d \neq \emptyset $. Indeed, clearly the long commutator $$[z^{-1}w_1,\dots,z^{-1}w_{d-1}] \coloneq [[[z^{-1}w_1,z^{-1}w_2],z^{-1}w_3],\dots,z^{-1}w_{d-1}]$$ belongs to $K_d$, and
its image in $F/L_d \cong \langle x, y \rangle \leqslant F$ is equal to the long commutator $[w_d^{-1}w_1,\dots,w_d^{-1}w_{d-1}]$. The latter is non-trivial in $\langle x,y \rangle \leqslant F$
by the choice of $w_1,\dots,w_d$ (because the elements $w_d^{-1}w_1,\dots,w_d^{-1}w_{d-1}$ freely generate a free subgroup of rank $d-1$),
hence  $[z^{-1}w_1,\dots,z^{-1}w_{d-1}] \in K_d \setminus L_d$ in $F$. In view of  the symmetry, we can deduce that
\begin{equation}\label{eq:K_i-L_i}
K_i \setminus L_i \neq \emptyset \mbox{ in $F$, for each } i \in \{1,\dots,d\}.
\end{equation}

Let $\varphi_i:F \to F/L_i$ denote the canonical epimorphism, and let $\varphi: F \to F/L_1 \times \dots \times F/L_d$ denote the product homomorphism, defined by the formula
$\varphi(f) \coloneq \bigl( \varphi_1(f),\dots,\varphi_d(f)\bigr)$ for all $f \in F$. Finally, we define $H_d \coloneq \varphi(F)$ as the image of $F$ in  $F/L_1 \times \dots \times F/L_d\cong \mathbb{F}_2^d$.

Notice that $\varphi_j(K_i)=\{1\}$ if $j \neq i$ by the definition of $K_i$, and $\varphi_i(K_i) \neq \{1\}$ by \eqref{eq:K_i-L_i} for any $i \in \{1,\dots,d\}$.
It follows that the subgroup $\varphi(K_1 \cdots K_d) \leqslant H_d$ is naturally isomorphic to the direct product $\varphi_1(K_1) \times\dots\times \varphi_d(K_d)$ of non-trivial free groups.
Consequently, $H_d$ contains a copy of $\Z^d$ as a subgroup, and so $\cd(H_d) \ge d$. On the other hand, $\cd(H_d) \le \cd(\mathbb{F}_2^d)=d$, hence $\cd(H_d)=d$, as claimed.
\end{proof}

\section{Finitely presented subdirect products of limit groups}
The goal of this section is to prove Theorem \ref{thm:adeq_free} and Corollaries \ref{cor:univ_for_free},\ref{cor:res_free} formulated in the Introduction.

Let $H$ be a subgroup of the direct product $P\coloneq F_1 \times \dots \times F_m$ of some groups $F_1, \dots,F_m$. Recall that $H$ is called a \emph{subdirect product} if it
projects surjectively onto each of the direct factors.
Following \cite{B-H-M-S} we will say that $H$ is \emph{virtually surjective on pairs} (VSP) if for any pair of distinct indices $i,j \in \{1,\dots,m\}$
the image of $H$ under the canonical projection $\rho_{ij}: P \to F_i \times F_j$ has finite index in $F_i \times F_j$. Let $N_i \coloneq H \cap \ker \rho_i \lhd H$, $i=1,\dots,m$,
where $\rho_i:P \to F_i$ denotes the canonical projection to the $i$-th coordinate.

\begin{lemma}\label{lem:VSP-reform} Let $H\leqslant   F_1 \times \dots \times F_m$ be a subdirect product. Then $H$ is VSP if and only if $|H:N_iN_j|<\infty$ for all $1 \le i <j \le m$.
\end{lemma}

\begin{proof} If $H$ is VSP then $|F_i \times F_j:\rho_{ij}(H)|<\infty$ whenever $i < j$, hence $F_i \cap \rho_{ij}(H)=\rho_{ij}(N_j)=\rho_i(N_j)$ has finite index in $F_i$.
Therefore $\rho^{-1}_i(\rho_i(N_j)) \cap H=N_iN_j$ has finite index in $H$.

On the other hand, if $|H:N_iN_j|<\infty$ for some $0 \le i <j \le m$, then $\rho_{i}(N_iN_j)=\rho_i(N_j)$ has finite index in $\rho_i(H)=F_i$. But $\rho_i(N_j)=\rho_{ij}(N_j) \subseteq \rho_{ij}(H)$ because
$\rho_{ij}(N_j)$ is contained in $F_i$ as $N_j \subseteq \ker \rho_j$. Thus $\rho_{ij}(H)$ contains a finite index subgroup of $F_i$. Similarly, it also contains $\rho_{j}(N_i)$, which has finite index in $F_j$. Hence
$|F_i \times F_j :\rho_{ij}(H)|<\infty$.
\end{proof}

\begin{rem}\label{rem:not_fp}
In \cite{B-M} Bridson and Miller proved that if a finitely presented subdirect product of free groups intersects each of the direct factors non-trivially
then it is VSP\footnote{This is implicit in the proof of \cite[Thm. 4.7]{B-M} and is explicitly stated in the Introduction of  \cite{B-H-M-S}, where it is proved that being VSP is also
sufficient for finite presentability of subdirect products of finitely presented groups.}. It follows that
for any $d \in \N$, the group $H_d$, constructed in the proof of Proposition~\ref{prop:H_d}, is not finitely presented.
Indeed,  if $H_d$ was VSP then  $|H_d/(N_iN_j)|<\infty$ for any $1 \le i < j \le d$ by Lemma~\ref{lem:VSP-reform}. But clearly $H/(N_iN_j) \cong F/(L_iL_j)$ is an infinite group because
it  has a presentation with $3$ generators and two relators (in particular, it surjects onto $\Z$).
\end{rem}

The following elementary observation will be useful.
\begin{lemma}\label{lem:image_of_prod} Suppose that $H$ and $F$ are groups and $\psi_i:H \to F$, $i=1,2$, are two epimorphisms. Assume that $H$ is generated by some subset $X \subseteq H$, and let
$M \lhd F$ denote the normal closure of the subset $\{ \psi_1^{-1}(x)\psi_2(x) \mid x \in X \}$. Then $\psi_1(\ker \psi_2) \subseteq M$; in particular, if one sets $L \coloneq \ker\psi_1\ker\psi_2 \lhd H$, then
$H/L$  maps onto $F/M$.
\end{lemma}

\begin{proof} If $h \in \ker\psi_2$, then $h=x_1\dots x_k$, for some $x_1,\dots,x_k \in X^{\pm 1}$, and $\psi_2(x_1) \dots \psi_2(x_k)=1$ in $F$. Since $\psi_2(x_j) \equiv \psi_1(x_j)$ (mod $M$), we see that
$\psi_1(h)=\psi_1(x_1)\dots\psi_1(x_k) \in M$. Thus $\psi_1(\ker\psi_2) \subseteq M$, as claimed.
\end{proof}

\begin{prop} \label{prop:vsp} There is a computable function $\gamma: \N \to \N$ such that the following holds. Let
$H$ be a group with a collection of normal subgroups $N_1,\dots,N_m \lhd H$, such that $H/N_i$ maps onto the free group of rank $2$ for each $i=1,\dots,m$, and
$|H:N_iN_j|<\infty$ for all $1\le i<j\le m$. If $H$ is $n$-generated, for some $n \in \N$, then $m \le \gamma(n)$.
\end{prop}

\begin{proof} Let $F$ denote the free group of rank $2$. By the assumptions, for every $i=1,\dots,m$, there exists an epimorphism $\psi_i: H \to F$ whose kernel contains $N_i$.

Suppose that some elements $x_1,\dots,x_n \in H$ generate $H$, and let $u_{ik} \coloneq \psi_i(x_k) \in F$, $k=1,\dots,n$, $i=1,\dots,m$. For any two distinct indices $i,j \in \{1,\dots,m\}$
let $L_{ij}\coloneq \ker\psi_i \ker\psi_j \lhd H$ and let $M_{ij}$ be the normal closure of the subset $\{u_{i1}^{-1}u_{j1},\dots,u_{in}^{-1}u_{jn}\}$ in $F$.
By construction, the group $H/(N_iN_j)$ maps onto the group $H/L_{ij}$, and the latter group maps onto the group $F/M_{ij}$ by Lemma \ref{lem:image_of_prod}. Since $|H/(N_iN_j)|<\infty$
by the assumptions, the group $F/M_{ij}$ must also be finite whenever $1 \le i<j\le m$.

Let $\{D_lF\}_{l=1}^\infty$ denote the Jennings-Zassenhaus filtration of the free group $F$ for the prime $p=2$. Namely, $D_lF\coloneq \{f \in F \mid f-1 \in I^n\}$, where $I$ is the augmentation ideal of the group ring
$\mathbb{K}[F]$ and  $\mathbb{K}$ is the field of two elements. Then $D_lF$ is a normal subgroup of finite index in $F$ for each $l \in \N$, and $F=D_1F \geqslant D_2F \geqslant \dots$.

Now, for the given $n \in \N$ choose $\tau \coloneq 2/3$ and $r\coloneq\lceil 2\log_2(3n)\rceil \in \N$ (so that $n \tau^r<2\tau-1$), and set $\gamma(n)\coloneq |F/D_rF|^n \in \N$.
Let us show that $m \le \gamma(n)$. Arguing by contradiction, suppose that $m>\gamma(n)$. Then, by the pigeon hole principle, there must exist a pair of indices $i,j \in \{1,\dots,m\}$, $i<j$, such that
$u_{ik} \equiv u_{jk}$ (mod $D_rF$) for all $k=1,\dots,n$. It follows that the elements $u_{i1}^{-1}u_{j1},\dots,u_{in}^{-1}u_{jn}$, normally generating $M_{ij}$, are all contained in $D_rF$.
Since $F$ is the free group of rank $2$ and
$$ 1-2 \tau +\sum_{i=1}^n \tau^{r}<0,$$ the Golod-Shafarevich theorem (see \cite{Ershov}) implies that the group $F/M_{ij}$ must be infinite.
This contradiction shows that $m \le \gamma(n)$, finishing the proof of the proposition.
\end{proof}

\begin{rem}\label{rem:gamma_estimate} We can estimate the value of $\gamma(n)$, for the choice of $\tau=2/3$, $r=\lceil 2\log_2(3n)\rceil$ and the field of two elements $\mathbb{K}$, as follows.
Let $Y$ be a generating set for the free group $F$ with $|Y|=2$.
Then, for any $l \in \N$, the vector space $I^l/I^{l+1}$, over $\mathbb K$, is spanned by the elements $(y_1-1)(y_2-1) \cdots(y_l-1)+I^{l+1}$, for all $y_1,\dots,y_l \in Y$ (see the proof of \cite[Lemma 8]{Grig}).
Hence $\dim_\mathbb{K}(I^l/I^{l+1}) \le 2^l$, and so $|I^l/I^{l+1}|\le 2^{2^l}$. Since $|D_lF/D_{l+1}F| \le |I^l/I^{l+1}|$ and $F=D_1F$, we obtain
$$ |F/D_rF|=|D_1F/D_2F| \cdots |D_{r-1}F/D_rF| \le 2^{2^{r}-1} \le 2^{18 n^2-1}. $$
Thus $\gamma(n) \le 2^{18n^3-n}$ for all $n \in \N$.

For any $\varepsilon>0$ by taking the parameter $\tau$ to be very close to $1/2$ (but still larger than 1/2) and choosing an appropriate $r$ we can get a better estimate $\gamma(n) \le 2^{Cn^{2+\varepsilon}-n}$,
where $C$ will be some large constant depending on $\tau$.
\end{rem}

Recall that a group $G$ is said to be \emph{residually free} if for every $g \in G\setminus \{1\}$ there exists a free group $F$ and a homomorphism $\phi:G \to F$ such that $\phi(g) \neq \{1\}$ in $F$.
The group $G$ is \emph{fully residually free} if for every finite subset $S \subseteq G$ there is a homomorphism from $G$ to a free group whose restriction to $S$ is injective.
Finitely generated fully residually free groups are commonly called \emph{limit groups}. These groups play an important role in the study of the first order theory of free groups (see \cite{Kh-M,Sela}).

In \cite{B-H-M-S} Bridson, Howie, Miller and Short proved that a finitely presented subgroup of a direct product of finitely many non-abelian limit groups, which is subdirect and intersects each of the direct factors non-trivially, must be VSP.
We will now use this result to prove Theorem \ref{thm:adeq_free}.

\begin{proof}[Proof of Theorem \ref{thm:adeq_free}] Let $\rho_i: F_1 \times \dots \times F_d \to F_i$ denote the canonical projection onto $F_i$. Since finitely generated subgroups of limit groups are also limit groups,
we can replace each $F_i$ by $\rho_i(H)$ to further assume that $H$ is subdirect.

After renumbering, if necessary, we can suppose that $F_1,\dots, F_s$ are abelian and $F_{s+1},\dots,F_d$ are non-abelian limit groups, for some $s \in \{0,\dots,d\}$.
Let $H_1$ and $H_2$ denote the images of $H$ under the canonical projections to $F_1 \times \dots \times F_s$ and $F_{s+1} \times \dots \times F_d$ respectively.
Since $\{1\} \neq F_i \cap H \subseteq F_i \cap H_1$ for all $i =1,\dots,s$, the rank $r$, of the  free abelian group $H_1$, must be at least $s$.
On the other hand, $r \le n$ because $H_1$ is $n$-generated, as a quotient of $H$, and the free abelian group of rank $r$ cannot be generated by fewer than $r$ elements. Hence $s \le n$.

Note that $H_2$ is finitely presented as a quotient of the finitely presented group $H$ by a finitely generated central subgroup $Z \coloneq H \cap (F_1 \times \dots \times F_s)$ ($F_1 \times \dots \times F_s$ is a
free abelian group of rank at most $ns$, so $Z$ can be generated by no more than $ns$ elements).  By construction, $H_2 \leqslant F_{s+1} \times \dots \times F_d$ is a subdirect product of non-abelian limit groups
intersecting each factor non-trivially. Therefore, by \cite[Thm. D(5)]{B-H-M-S}, $H_2$ must be VSP in $F_{s+1} \times \dots \times F_d$.

From Lemma \ref{lem:VSP-reform} it follows that $|H_2:N_iN_j|<\infty$ for all $s+1 \le i<j \le d$, where $N_i\lhd H_2$ denotes the kernel of the projection onto $F_i$, $i=s+1,\dots,d$.
Since $H_2/N_i \cong F_i$ is a non-abelian limit group for $i=s+1,\dots,d$, it has an epimorphism onto $\F_2$, so all the assumptions of Proposition~\ref{prop:vsp} are satisfied, and we can use it to conclude that $d-s \le \gamma(n)$.
Thus,  $d \le \delta(n)$, where $\delta(n) \coloneq n +\gamma(n)$ for all $n \in \N$, and the theorem is proved.
\end{proof}

\begin{rem}\label{rem:delta_estimate} In view of Remark \ref{rem:gamma_estimate} we can estimate that $\delta(n)\le n+ 2^{18n^3-n} <2^{18n^3}$ for all $n \in \N$.
\end{rem}

If a subgroup $H$, of a direct product $F_1 \times \dots \times F_m$, has trivial intersection with one of the factors, then $H$ is isomorphic to its image under the projection away from this factor.
So we can always find an embedding of $H$ into $F_{i_1} \times \dots \times F_{i_k}$, for some $\{i_1,\dots,i_k\} \subseteq \{1,\dots,m\}$, such that $H \cap F_{i_j} \neq \{1\}$ for every $j=1,\dots,k$.
In view of Theorem \ref{thm:adeq_free}, this observation together with the fact that every finitely generated free group embeds into $\F_2$ yields Corollary~\ref{cor:univ_for_free}. Corollary \ref{cor:res_free} is
obtained similarly,  by first embedding the finitely generated residually free group $G$ into the direct product of finitely many limit groups (which can be done by \cite[Cor. 19]{B-M-R}; see also \cite[Cor. 2]{Kh-M} or \cite[Claim 7.5]{Sela}),
and then removing the redundant direct factors that intersect the image of $G$ trivially.

\section{Background on right angled Artin groups}\label{sec:b-n}
Let $\ga$ be a finite simplicial graph with vertex set $V\ga=V$ and the edge set $E\ga=E$. For any vertex $v \in V$,  the \emph{link} $\link_\ga(v) \subseteq V$ is defined as the set of all vertices of $\ga$ adjacent to $v$
(not including $v$ itself). If $U \subseteq V$ then $\link_\ga(U) \coloneq \bigcap_{u \in U} \link_\ga(u)$. A subset $U \subseteq V$ is \emph{reducible} if it can be decomposed in a disjoint union $U=X \sqcup Y$ of non-empty subsets
$X,Y \subset V$ such that $Y  \subseteq \link_\ga(X)$. If there is no such decomposition, then $U$ is said to be {\it irreducible}.

Let $A=A(\ga)$ be the right angled Artin group, associated to $\ga$.
Then the elements of $V$ can be thought of as generators of $A$, and every element $g \in A$ can be represented by a word $w$ over the alphabet $V^{\pm 1}$. The \emph{support} of $w$
is the set of all vertices $v \in V$ such that either $v$ or $v^{-1}$ appears in $w$. The \emph{length} $|g|_\ga$ and the \emph{support} $\supp_\ga(g)$, of $g \in A$, are defined as the length and the
support of a shortest word over $V^{\pm 1}$ representing $g$ in $A$, respectively. It is a standard fact that $|g|_\ga$ and $\supp_\ga(g)$ only depend on $g$ and are independent of the choice of a shortest word representing it.

Right angled Artin groups are special cases of \emph{graph products of groups}, when all the vertex groups are infinite cyclic (see \cite[Subsec. 2.2]{A-M}
for some background on graph products).

Given any subset $U \subseteq V\ga$, the subgroup $A_U \leqslant A$, generated by $U$, is called a \emph{full subgroup} (or a \emph{special subgroup}) of $A$.
It is not difficult to see that $A_U$ is naturally isomorphic to the right angled Artin group
corresponding to the full subgraph $\ga_U$ of $\ga$, spanned by the vertices from $U$. Moreover, $A_U$ is a retract of $A$, where the \emph{canonical retraction} $\rho_U:A \to A_U$ is defined on the generating set $V$ of $A$
by the formula  $\rho_U(u)=u$ if $u \in U$ and $\rho_U(v)=1$ if $v \in V \setminus U$.

Any conjugate of a full subgroup in $A$ is said to be \emph{parabolic}. Every subset $M \subseteq A$ is contained in a
unique minimal (with respect to inclusion) parabolic subgroup $\pc_\ga(M)$, which is called the \emph{parabolic closure} of $M$ in $A$ (see \cite[Prop. 3.10]{A-M}).

An element $t \in A$ is \emph{cyclically reduced} if it has minimal length in its conjugacy class in $A$. Servatius \cite[Prop. on p. 38]{Serv} proved that for every element $g \in A$
there exist $u,t \in A$ such that $g=utu^{-1}$, $t$ is cyclically reduced and $|g|_\ga=|t|_\ga+2|u|_\ga$.

If $t \in A$ and $H \leqslant A$ then we will write $\C_H(t)$ and $\No_A(H)$ to denote the centralizer of $t$ in $H$ and the normalizer of $H$ in $A$ respectively.

\begin{lemma} \label{lem:centr_of_powers} If $A$ is a right angled Artin group then for any $t \in A$ and $m \in \Z\setminus\{0\}$,  $\C_A(t^m)=\C_A(t)$.
\end{lemma}

\begin{proof} Suppose that $s t^m s^{-1}=t^m$ for some $s \in A$. Then $(sts^{-1})^m=t^m$, and since right angled Artin groups  have the Unique Root property (see, for example, \cite[Lemma 6.3]{M-RAAG}),
we can deduce that $sts^{-1}=t$. Thus $\C_A(t^m) \subseteq \C_A(t)$. The reverse inclusion is obvious, hence the lemma is proved.
\end{proof}

\section{Special HNN-extensions in right angled Artin groups}\label{sec:5}
In this section we give a criterion for embedding the special HNN-extension of a group with respect to a normal subgroup into a right angled Artin group. The next lemma uses the obvious fact that the class of right angled Artin groups
contains the infinite cyclic group and is closed under forming free and direct products.
\begin{lemma}\label{lem:spec-HNN-emb} Suppose that $F$ is a subgroup of some right angled Artin group $A$, $N \lhd F$ and $F/N$ is embeddable into some right angled Artin group $B$.
Then the special HNN-extension \eqref{eq:spec-HNN} can be embedded into the right angled Artin group $C\coloneq A \times (B*\langle t \rangle)$, where $\langle t \rangle$ denotes the infinite cyclic group  generated by $t$.
\end{lemma}

\begin{proof} Abusing the notation, let us assume that $F \leqslant A$ and $F/N \leqslant B$. Let $\phi:F \to B$ denote the natural homomorphism with $\ker(\phi)=N$ and $\mathrm{im}(\phi)=F/N$.
Clearly the subgroup $\bF\coloneq \{(f,\phi(f)) \mid f \in F\}\leqslant A \times B$ is isomorphic to $F$ and $\bF \cap A=N$. Now, let $\bG$ be the subgroup of $C= A \times (B*\langle t \rangle)$ generated by $\bF$ and $t$.

Evidently, $t$ commutes with $(a,b) \in A \times B$ in $C$ if and only if $b=1$, thus $\C_{\bF}(t)=\bF \cap A=N$. This naturally gives rise to a homomorphism $\psi: G \to \bG$, defined according to
the formula $\psi(f)\coloneq (f,\phi(f))\in \bF$ and $\psi(t)=t$, where $G$ is the special HNN-extension \eqref{eq:spec-HNN}, of $F$ with respect to $N$. Obviously $\psi$ is surjective, so it remains to show that it is injective.
This can be easily done by using normal forms: suppose that $g \in G \setminus \{1\}$. Then $g=f_0t^{k_1}f_1 \cdots t^{k_n}f_n$, where $n \ge 0$,
$f_0, \dots,f_n \in F$, $k_1,\dots,k_n \in \Z \setminus\{0\}$, $f_i \notin N$ for $1 \le i \le n-1$, and $f_0 \neq 1$ if $n=0$.

If $n=0$ then $g=f_0\neq 1$ in $F$, hence $\psi(g)=(f_0,\phi(f_0)) \neq 1$ in $A \times B \leqslant C$.
So, suppose that $n \ge 1 $ and let $\rho: C \to B*\langle t \rangle$ denote the canonical retraction.
Then $\rho(\psi(g))=\phi(f_0)t^{k_1} \phi(f_1) \cdots t^{k_n} \phi(f_n)$ is a non-empty reduced word in the free product $B * \langle t \rangle$, hence $\rho(\psi(g)) \neq 1$.
Therefore $\psi(g) \neq 1$ in $C$, and  $\psi$ is injective.

Thus  $G \cong \bG \leqslant C$, and the lemma is proved.
\end{proof}

Let $A$ be a right angled Artin group associated to a finite graph $\ga$ with vertex set $V=V\ga$. Recall that for any subset $M \subseteq A$, $\pc_\ga(M)$ denotes the
parabolic closure of $X$ in $A$ (see Section~\ref{sec:b-n}).

\begin{prop}\label{prop:emb_quot} Let $F$ be a subgroup of $A$, and let $t \in A$ be an element such that $\pc_\ga(G)=A$, where $G\coloneq \langle F,t \rangle \leqslant A$. Suppose that
$N\coloneq \C_F(t)$ is a normal subgroup of $F$. Then
\begin{itemize}
  \item[(i)] $\pc_\ga(N)$ is a direct factor of $A$. More precisely, $\pc_\ga(N)=A_X$ for some $X \subseteq V$, and $V=X \sqcup \link_\ga(X)$.
  \item[(ii)] $F \cap A_X=N$, hence $F/N$ embeds into $A_Y \leqslant A$, where $Y \coloneq V \setminus X=\link_\ga(X)$.
\end{itemize}
\end{prop}

\begin{proof} Clearly $N$ is normal in $G$, hence $G \leqslant \No_A(N) \leqslant \No_A(\pc_\ga(N))$, where the second inclusion was proved in \cite[Lemma 3.12]{A-M}.
Suppose that $\pc_\ga(N)=h A_X h^{-1}$ for some $X \subseteq V$ and some $h \in A$. By \cite[Prop. 3.13]{A-M},
$\No_A(hA_Xh^{-1})=hA_{X \cup \link_\ga(X)}h^{-1}$ is a parabolic subgroup of $A$ containing $G$.
Since $\pc_\ga(G)=A$, it follows that $A=hA_{X \cup \link_\ga(X)}h^{-1}=A_{X \cup \link_\ga(X)}$. Therefore $A=A_X \times A_{\link_\ga(X)}$,
$\pc_\ga(N)=A_X$ and $V=X \sqcup Y$, where $Y=\link_\ga(X)$. Thus (i) is proved.

Let us prove (ii) now.
Clearly, after replacing $t$, $F$ and $N$ with their conjugates by the same element, we can assume that $t$ is cyclically reduced in $A$ (note that this does not affect the equality $\pc_\ga(N)=A_X$, as $A_X \lhd A$).
Let $\supp_\ga(t)=T_1 \sqcup \dots \sqcup T_k$ be a decomposition into the disjoint union of non-empty irreducible subsets $T_i$ of $V$, such that $T_i \subseteq \link_\ga(T_j)$ in $\ga$ whenever $i \neq j$.
Then one can write $t=t_1^{n_1} t_2^{n_2} \cdots t_k^{n_k}$, where $t_i \in A_{T_i}\setminus\{1\}$ are not proper powers and $n_i \in\N$ for $i=1,\dots, k$.

In \cite{Serv} Servatius showed that the centralizer, $\C_A(t)$, is generated by $t_1,\dots,t_k$ and $S\coloneq \link_\ga(\supp_\ga(t)) \subseteq V$. In other words,
\begin{equation}\label{eq:centralizer}
\C_A(t)=\langle t_1\rangle \langle t_2 \rangle \cdots \langle t_k \rangle A_S \cong \langle t_1\rangle \times \langle t_2 \rangle  \times \dots  \times \langle t_k \rangle  \times A_S.
\end{equation}

For every $i \in \{1,\dots,k\}$, let $\rho_i:A \to A_{T_i}$ denote the canonical retraction.
Suppose that $\rho_j(N) \neq \{1\}$ for some $j \in \{1,\dots,k\}$.
Evidently, since $N \leqslant \C_A(t)$, we have $\rho_j(N) \leqslant \rho_j(C_A(t))=\langle t_j \rangle$ by \eqref{eq:centralizer}.
Consequently, $\rho_j(N)=\langle t_j^{m_j} \rangle$
for some $m_j \in \Z\setminus\{0\}$. Recalling that $N \lhd F$, we see that $\rho_j(N) \lhd \rho_j(F)$. Since $\rho_j(N) \cong \Z$, as right angled Artin groups are torsion-free, we can use Lemma \ref{lem:centr_of_powers}
to conclude that $\rho_j(N)$ is actually central in  $\rho_j(F)$ (see \cite[Lemma 9.1]{A-M-S} for an alternative proof), thus $\rho_j(F) \leqslant \C_{A_{T_j}}(t_j^{m_j})$. Therefore,
applying Lemma~\ref{lem:centr_of_powers} one more time, we see that $\rho_j(F) \leqslant\C_{A_{T_j}}(t_j)=\langle t_j \rangle$ in $A_{T_j}$.

After renumbering, if necessary, we can assume that there exists $l \in \{0,1, \dots,k\}$ such that $\rho_i(N) \neq \{1\}$ if $1 \le i \le l$, and $\rho_i(N)=\{1\}$ if $l+1 \le i \le k$.
The inclusion $N \leqslant \C_A(t)$, together with \eqref{eq:centralizer}, implies that $\pc_\ga(N) \leqslant A_Z$, where $Z \coloneq \bigcup_{j=1}^l T_j \cup S \subseteq V$.
In the previous paragraph we have shown that $\rho_j(F) \leqslant \langle t_j \rangle$ whenever $1 \le j \le l$. Let $\rho_S:A \to A_S$ denote the canonical retraction. Since
$A_Z$ is canonically isomorphic to the direct product $A_{T_1} \times \dots \times A_{T_l} \times A_S$,  for any $g \in F \cap A_Z$ we have
$$g=\rho_1(g) \cdots \rho_l(g) \rho_S(g) \in \langle t_1 \rangle \cdots \langle t_l \rangle A_S \subseteq \C_A(t).$$ Therefore, $F \cap A_Z \subseteq F \cap \C_A(t)=N$ by the assumption,
yielding that $F \cap A_X=N$, as $A_X=\pc_\ga(N) \leqslant A_Z$. Thus the proof of claim (ii) is complete.
\end{proof}

Special HNN-extensions play an important role in group theory. For example,  Higman used them in the proof of his famous embedding theorem -- see \cite{Higman} and \cite[Ch. IV.7]{L-S}. In particular
he showed that if $F$ is a finitely generated group and $N \lhd F$ is such that the special HNN-extension \eqref{eq:spec-HNN} can be embedded into a finitely presented group then the quotient $F/N$ can also be
embedded into a finitely presented group. Proposition \ref{prop:emb_quot} allows us to obtain the following analogue of the latter statement, where the class of finitely presented groups is replaced by the class of
right angled Artin groups.

\begin{cor}\label{cor:rope-RAAG} Suppose that $F$ is a group, $N \lhd F$ and $G$ is the special HNN-extension \eqref{eq:spec-HNN}, of $F$ with respect to $N$. Then the following are equivalent:
\begin{itemize}
\item[(a)] $F \in \cA$ and $F/N \in \cA$;
\item[(b)] $G \in \cA$.
\end{itemize}
\end{cor}

\begin{proof} The implication (a)$\Rightarrow$(b) is given by Lemma \ref{lem:spec-HNN-emb}. To prove the reverse implication, assume that $G \hookrightarrow B$ for some right angled Artin group $B$, corresponding to a
finite simplicial graph $\ga'$. Evidently, since $F \leqslant G$, we have $F \in \cA$.

Note that  $A \coloneq \pc_{\ga'}(G)$ is a right angled Artin group, corresponding to some full subgraph $\ga$ of $\ga'$. Moreover, the parabolic closure $\pc_{\ga}(G)$, of $G$ in $A$, is $A$ itself.
This is because every parabolic subgroup of $A$ is a parabolic subgroup of $B$ and $A$ is the minimal parabolic subgroup of $B$ containing $G$, by definition.
Britton's lemma for HNN-extensions (see \cite[Sec. IV.2]{L-S}) easily yields that $C_F(t)=N$.
Therefore all the assumptions of Proposition \ref{prop:emb_quot} are satisfied and we can deduce that $F/N \hookrightarrow A$. Thus (b) implies (a).
\end{proof}

We are now ready to prove Theorem \ref{thm:rope} from the Introduction.

\begin{proof}[Proof of Theorem \ref{thm:rope}] Claim (i) is an immediate consequence of Lemma \ref{lem:spec-HNN-emb}. Suppose that $G \hookrightarrow A$ for some right angled Artin group $A$.
Abusing the notation, let us identify $G$ with its isomorphic image in $A$.  Arguing as in the proof of Corollary \ref{cor:rope-RAAG}, we can replace $A$ with a parabolic subgroup to ensure that $A$ is the parabolic closure of $G$.
As before, we know that $C_F(t)=N$ by Britton's lemma, hence $F/N$ embeds in $A$  by  Proposition \ref{prop:emb_quot}, and claim (ii) holds.
\end{proof}

It is not difficult to see that the group $G_d$ from Corollary \ref{cor:G_d}, formulated in the Introduction, is not finitely presented for any $d \in \N$.
This is because the normal subgroup $N \lhd F$, with $F/N \cong H_d$, cannot be finitely generated
(as it is both infinite and has infinite index in the free group $F$ -- see \cite[Prop. I.3.12]{L-S}), and the special HNN-extension \eqref{eq:spec-HNN} is finitely presented if and only if $F$ is finitely presented
and $N$ is finitely generated. One can overcome this obstacle by using the famous construction of Rips \cite{Rips}, where the free group is replaced by a hyperbolic group.

\begin{cor}\label{cor:fp_ex} For every $d \in \N$ there is a finitely presented group $P_d \in \cA$ such that $\cd(P_d) \le 3$ but any right angled Artin group $A$, with $P_d \hookrightarrow A$,
satisfies $\cd(A)\ge d$.
\end{cor}

\begin{proof} Take any $d \in \N$. Let us use the modification of Rips's construction suggested by Haglund and Wise in \cite{H-W}
(in view of the results of Wise \cite[Thm.~1.2]{Wise} and  Agol \cite[Thm.~1.1]{Agol} we can also use the original construction of Rips from \cite{Rips}).
Namely, according to \cite[Thm.~10.1]{H-W},  there is a group $S$ with a normal subgroup $K \lhd S$ such that $S$ is the fundamental group of some compact non-positively curved square complex,
$K$ is finitely generated and $S/K \cong \Z^d$. Moreover, Theorems 5.5, 5.7 and 4.2 from \cite{H-W} imply that $S$ contains a finite index subgroup $F$ such that $F \in \cA$.

Let $N \lhd F$ be the intersection of $F$ and $K$. Note that $N$ is finitely generated as $|K:N|\le |S:F|<\infty$,
and  $F/N$ is still isomorphic to the free abelian group of rank $d$ (because it has finite index in $S/K $). Now, let $P_d$ be the special HNN-extension
of $F$ with respect to $N$. By Theorem \ref{thm:rope}, $P_d \hookrightarrow \cA$ and for any right angled Artin group $A$, containing $P_d$, one has $\Z^d \hookrightarrow A$, hence $\cd(A) \ge d$.

Let  $X \subseteq N$ be some finite generating set of $N$. Then $P_d$ can be defined by the (relative) presentation
\begin{equation}\label{eq:pres}
P_d=\langle F,t \,\|\, txt^{-1}=x,~\forall\,x \in X \rangle.
\end{equation}
Observe that $S$ is finitely presented, as the fundamental group of a finite square complex, hence $F$ is finitely presented as well. Equation \eqref{eq:pres} shows that a presentation for $P_d$ can be obtained from a finite presentation
of $F$ by adding one generator and finitely many defining relations, thus $P_d$ is also finitely presented.

Finally, $\cd(S) \le 2$ as $S$ is the fundamental group of a $2$-complex with contractible universal cover, therefore $\cd(F)\le 2$. And since $P_d$ is an HNN-extension of $F$, we can conclude that
$\cd(P_d) \le \cd(F)+1=3$  -- see  \cite[Ex. 4 in Ch. VIII.2]{Brown}.
\end{proof}

Corollary \ref{cor:fp_ex} demonstrates that there is no ``universal'' right angled Artin group for the class of all finitely presented groups $P \in \cA$ with $\cd(P) \le 3$.
However, the main result of Kim and Koberda from \cite{K-K-obstr} implies that an even stronger fact is true: there does not exist a right angled Artin group which contains copies of all
$2$-dimensional right angled Artin groups (see also \cite[Thm. 1.16]{K-K} for a weaker statement). Indeed, one just needs to recall that
every right angled Artin group embeds into the mapping class group Mod($\Sigma$), of some orientable compact surface $\Sigma$ (cf. \cite[Prop. 1.3]{Koberda}), and combine this with \cite[Thm. 1.2]{K-K-obstr},
asserting that for every such surface $\Sigma$ there exists a right angled Artin group $A$, with $\cd(A)=2$, such that $A \not\hookrightarrow \mbox{Mod($\Sigma$)}$.

\medskip
\noindent {\bf Acknowledgements.} The author would like to thank Montserrat Casals-Ruiz, Fran\c cois Dahmani and Nansen Petrosyan for discussions, and Alexander Olshanskii for suggesting the idea to use
the Golod-Shafarevich theorem in the proof of Proposition \ref{prop:vsp}. The author is also grateful to the referee for bringing the paper \cite{K-K-obstr} to his attention.


\begin{thebibliography}{99}
\bibitem{Agol} I. Agol,
The virtual Haken conjecture. With an appendix by I. Agol, D. Groves and J. Manning.
{\it Documenta Math}. \textbf{18} (2013), 1045--1087.

\bibitem{A-M}
Y. Antol\'{i}n, A. Minasyan,
Tits alternatives for graph products.
\emph{J. Reine Angew. Math.}, to appear.

\bibitem{A-M-S} Y. Antol\'{i}n, A. Minasyan, A. Sisto, Commensurating endomorphisms of acylindrically hyperbolic groups and applications. Preprint (2013).\texttt{ arXiv:1310.8605}

\bibitem{Baud} A. Baudisch, Subgroups of semifree groups.
\emph{Acta Math. Acad. Sci. Hungar.} \textbf{38} (1981), no. 1-4, 19--28.

\bibitem{B-M-R} G. Baumslag, A. Myasnikov, V. Remeslennikov, Algebraic geometry over groups. I. Algebraic sets and ideal theory.
\emph{J. Algebra} \textbf{219} (1999), no. 1, 16--79.

\bibitem{B-H-M-S}  M.R. Bridson, J. Howie, C.F. Miller III, H. Short,  On the finite presentation of subdirect products and the nature of residually free groups. \emph{Amer. J. Math.} \textbf{135} (2013), no. 4, 891--933.

\bibitem{B-M} M.R. Bridson, C.F. Miller III, Structure and finiteness properties of subdirect products of groups. \emph{Proc. Lond. Math. Soc. (3)} \textbf{98} (2009), no. 3, 631--651.

\bibitem{Brown} K.S. Brown, Cohomology of groups. Graduate Texts in Mathematics, 87. \emph{Springer-Verlag, New York-Berlin}, 1982. x+306 pp.

\bibitem{C-K} M. Casals-Ruiz, I. Kazachkov, Limit groups over partially commutative groups and group actions on real cubings. \emph{Geom. Topol.}, to appear.

\bibitem{Charney}  R. Charney, An introduction to right-angled Artin groups. \emph{Geom. Dedicata} \textbf{125} (2007), 141--158.

\bibitem{Ershov}  M. Ershov, Golod-Shafarevich groups: a survey. \emph{Internat. J. Algebra Comput.} \textbf{22} (2012), no. 5, 1230001.

\bibitem{Grig} R. Grigorchuk, On the Hilbert-Poincar\'e series of graded algebras that are associated with groups. (Russian) \emph{Mat. Sb.} \textbf{180} (1989), no. 2, 207--225, 304. English translation in
\emph{Math. USSR-Sb.} \textbf{66} (1990), no. 1, 211--229.

%

\bibitem{H-W} F. Haglund, D.T. Wise, Special cube complexes.
\emph{Geom. Funct. Anal.} \textbf{17} (2008), no. 5, 1551--1620.


\bibitem{Higman} G. Higman, Subgroups of finitely presented groups.
\emph{Proc. Roy. Soc. Ser. A} \textbf{262} (1961),  455--475.

\bibitem{H-N-N} G. Higman, B.H. Neumann, H. Neumann, Embedding theorems for groups.
\emph{J. London Math. Soc.} \textbf{24}, (1949). 247--254.

\bibitem{Kh-M} O. Kharlampovich, A. Myasnikov, Irreducible affine varieties over a free group. II. Systems in triangular
quasi-quadratic form and description of residually free groups. \emph{J. Algebra} \textbf{200} (1998), no. 2, 517--570.

\bibitem{K-K-obstr} S-h. Kim, T. Koberda, An obstruction to embedding right-angled Artin groups in mapping class groups, \emph{Int. Math. Res. Notices} \textbf{14} (2014), no 14, 3912--3918.

\bibitem{K-K} S-h. Kim, T. Koberda,  Embedability between right-angled Artin groups.
\emph{Geom. Topol.} \textbf{17} (2013), no. 1, 493--530.

\bibitem{Koberda} T. Koberda, Right-angled Artin groups and a generalized isomorphism problem for finitely generated subgroups of mapping class groups.
\emph{Geom. Funct. Anal.} \textbf{22} (2012), no. 6, 1541--1590.

\bibitem{L-S} R.C. Lyndon, P.E. Schupp, Combinatorial group theory. Ergebnisse der Mathematik und ihrer Grenzgebiete,
Band 89. \emph{Springer-Verlag, Berlin-New York}, 1977. xiv+339 pp.

\bibitem{M-RAAG} A. Minasyan, Hereditary conjugacy separability of right angled Artin groups and its applications. \emph{Groups Geom. Dyn.} \textbf{6} (2012), no. 2, 335--388.

\bibitem{Rips}  E. Rips, Subgroups of small cancellation groups. \emph{Bull. London Math. Soc.} \textbf{14} (1982), no. 1, 45--47.

\bibitem{Sela} Z. Sela, Diophantine geometry over groups. I. Makanin-Razborov diagrams.  \emph{Publ. Math. Inst. Hautes \'Etudes Sci.} No. 93 (2001), 31--105.

\bibitem{Serv} H. Servatius, Automorphisms of graph groups. \emph{J. Algebra} \textbf{126} (1989), no. 1, 34--60.

\bibitem{Wise}  D.T. Wise, Cubulating small cancellation groups. \emph{Geom. Funct. Anal.} \textbf{14} (2004), no. 1, 150--214.

\bibitem{Wise-QH} D.T. Wise, The structure of groups with a quasiconvex hierarchy.
Preprint (2011). Available from \\
\url{http://www.math.mcgill.ca/wise/papers.html}
\end{thebibliography}
\end{document}